\documentclass[pdftex,11pt]{amsart}%
\usepackage{amsmath}%
\usepackage{amssymb}%
\usepackage{amscd}%
\usepackage{color}%
\usepackage{enumerate}
\usepackage[all]{xy}%
\usepackage{newtxtext,newtxmath}%
\usepackage[all, warning]{onlyamsmath}
\usepackage[alphabetic]{amsrefs}
\usepackage{comment}
\usepackage[pdftex]{graphicx}
\usepackage{mathtools}
\usepackage{float}
\usepackage[normalem]{ulem}
\usepackage{comment}
\usepackage{bbm}
\usepackage{cleveref}
\usepackage{tikz}
\usetikzlibrary{arrows.meta}
\usetikzlibrary{arrows}
\usetikzlibrary {bending}
\usepackage{caption}
\usepackage{subcaption}
\def\mrm#1{\mathrm{#1}}
\def\mbb#1{\mathbb{#1}}

\theoremstyle{plain}
    \newtheorem{theorem}{Theorem}[section]
    \newtheorem{proposition}[theorem]{Proposition}
    \newtheorem{lemma}[theorem]{Lemma}

\theoremstyle{definition}
    \newtheorem{definition}[theorem]{Definition}

    \newtheorem{remark}[theorem]{Remark}

\def\Alphabet{A,B,C,D,E,F,G,H,I,J,K,L,M,N,O,P,Q,R,S,T,U,V,W,X,Y,Z}
\def\alphabet{a,b,c,d,e,f,g,h,i,j,k,l,m,n,o,p,q,r,s,t,u,v,w,x,y,z}
\def\endpiece{xxx}
\def\makeAlphabet[#1]{\expandafter\makeA#1,xxx,}
\def\makealphabet[#1]{\expandafter\makea#1,xxx,}
\def\makeA#1,{\def\temp{#1}\ifx\temp\endpiece\else%
\mkbb{#1}\mkfrak{#1}\mkbf{#1}\mkcal{#1}\mkscr{#1}\mkbs{#1}\expandafter\makeA\fi}%
\def\makea#1,{\def\temp{#1}\ifx\temp\endpiece\else\mkfrak{#1}\mkbf{#1}\mkbs{#1}\expandafter\makea\fi}%
\def\mkbb#1{\expandafter\def\csname bb#1\endcsname{\mathbb{#1}}}
\def\mkfrak#1{\expandafter\def\csname fr#1\endcsname{\mathfrak{#1}}}
\def\mkbf#1{\expandafter\def\csname b#1\endcsname{\mathbf{#1}}}
\def\mkcal#1{\expandafter\def\csname c#1\endcsname{\mathcal{#1}}}
\def\mkscr#1{\expandafter\def\csname s#1\endcsname{\mathscr{#1}}}
\def\mkbs#1{\expandafter\def\csname bs#1\endcsname{{\boldsymbol{#1}}}}
\def\makeop[#1]{\xmakeop#1,xxx,}
\def\mkop#1{\expandafter\def\csname #1\endcsname{{\mathrm{#1}}}} %
\def\xmakeop#1,{\def\temp{#1}\ifx\temp\endpiece\else\mkop{#1}\expandafter\xmakeop\fi}%
\def\makeup[#1]{\xmakeup#1,xxx,}
\def\mkup#1{\expandafter\def\csname #1\endcsname{{\mathrm{#1}\,}}} %
\def\xmakeup#1,{\def\temp{#1}\ifx\temp\endpiece\else\mkup{#1}\expandafter\xmakeup\fi}%
\makeAlphabet[\Alphabet]
\makealphabet[\alphabet]
\makeop[Hom,Sym,Fun,Map,An,Betti,id,Id,Th,Ind,coInd,Consv]
\makeop[rig,an,eff,op,aff,usu,big,vbig,fil,lft,cons,lis,CCart,qcqs,ayo,Ex,prop,coCart,smooth,cp,nil,all]
\makeop[Shv,PShv,DA,SH,AnSH,Cond,Rep,CAlg,Alg,Mod,Ab,Sch,Sp,Sm,Pr,AnSm,AnSpc,Perf,Ouv,Comp,Corr]
\makeup[Spec,Spa,Spv]

\begin{document}
\title{Decision problem on interactions}
\author[Wachi]{Hidetada Wachi}
\email[Wachi]{wachi213@keio.jp}
\address[Wachi]{Department of Mathematics, Faculty of Science and Technology, Keio University, 3-14-1 Hiyoshi, Kouhoku-ku, Yokohama 223-8522, Japan}

\date{\today}

\date{\today \quad ver.1.0.1}
\begin{abstract}
An interaction is a certain symmetric graph that describes the possible transition of states of adjacent sites of large-scale interacting systems. 
In the series of studies Bannai-Kametani-Sasada \cite{BKS}, Bannai-Sasada \cite{BS3}, they defined the notion of the irreducibly quantified interactions which is suitable for considering the hydrodynamic limits via the conserved quantities.  
In this paper, we prove that the property that an interaction is irreducibly quantified is decidable. 
\end{abstract}


\maketitle

%
\section{Introduction}
%

The hydrodynamic limit is a mathematical method that derives macroscopic deterministic partial differential equations as limits from microscopic large-scale interacting systems. 
Varadahn's decomposition theorem plays an important role in proving the hydrodynamic limit for the non-gradient model. 
In the series of studies Bannai-Kametani-Sasada \cite{BKS} and Bannai-Sasada \cite{BS3}, they prove Varadahn's decomposition theorem for a general class of large-scale interacting systems. 
In their setting, microscopic large-scale interacting systems are constructed by interactions. 
In the recent study \cite{BS:Unif}, Bannai-Sasada generalizes the definition of interaction and computes the $0$-th uniform cohomology. 

\begin{definition}[{\cite{BS:Unif}*{Definition 1.2}}]
Let $S$ be a finite set. 
We define a \textit{interaction} on $S$ to be a pair $(S,\phi)$ consisting of a subset $\phi\subset (S\times S)\times (S\times S)$ such that $(S\times S,\phi)$ form a symmetric digraph. 
We call $S$ the \textit{state space}. 
\end{definition}

In the theory of hydrodynamic limits, we are interested in conserved quantities of the interaction. 
A map $\xi:S\rightarrow \mathbb{R}$ is a \textit{conserved quantity} if for any $(s,t)$, $(s',t')\in S\times S$ such that $((s,t),(s',t'))\in \phi$,  we have $\xi(s)+\xi(t)=\xi(s')+\xi(t')$. 
Bannai-Koriki-Sasada-Wachi-Yamamoto \cite{BKSWY} provide classifications of interactions with respect to the space of conserved quantities. 

In the studies \cite{BKS} and \cite{BS3}, they consider the condition of interactions, which is called \textit{irreducibly quantified}, as the suitable assumption for the hydrodynamic limit.
Let $(X,E)$ be a finite connected symmetric digraph. 
In this paper, we assume that every graph has no loops or multiple edges.
For any conserved quantity $\xi$ of the interactions $(S,\phi)$, we define a function $\xi_X:S^X\rightarrow \mbb{R}$ by 
\begin{equation*}
    \xi_X(\eta)=\sum_{x\in X}\xi(\eta_x),
\end{equation*}
where $\eta=(\eta_x)\in S^X$.
We define a subset $\Phi_E\subset S^X\times S^X$ by 
\begin{equation*}
   \Phi_E=\{(\eta,\eta')\in S^X\times S^X\mid \exists (x,y)\in E,\,(\eta_x,\eta_y),(\eta'_x,\eta'_y))\in \phi,\,\forall z\neq x,y,\,\eta_z=\eta'_z\}.
\end{equation*}
We note that $(S^X,\Phi_E)$ is a symmetric digraph. 
Then we define irreducibly quantified interactions as follows. 

\begin{definition}[{\Cref{def:irreq}}]
We say an interaction $(S,\phi)$ is \textit{irreducibly quantified} if for any finite connected symmetric digraph $(X,E)$ and any $\eta,\eta'\in S^X$, if $\xi_X(\eta)=\xi_X(\eta')$ for any conserved quantity $\xi$ of $(S,\phi)$, then $\eta$ and $\eta'$ are in the same connected component of $(S^X,\Phi_E)$.
\end{definition}

A priori, deciding whether a given interaction is irreducibly quantified may require checking infinitely many conditions. 
However, in this paper, we prove that this can be determined in a finite number of steps.
More precisely, the main result of this paper is the following. 

\begin{theorem}[{\Cref{thm:decidable}}]\label{thm:decidable-intro}
The question of whether an interaction is irreducibly quantified is \textbf{decidable}. 
\end{theorem}

To prove our main theorem, we reduce the problem to that of semigroups. 
A semigroup $M=(M,\cdot)$ is a pair consisting of a set $M$ and a binary operation $\cdot:M\times M\rightarrow M$ which satisfies associativity $(a\cdot b)\cdot c=a\cdot (b\cdot c)$ for any $a,b,c\in M$.
We will apply the result of Narendran-\'{O}'D\'{u}nlaing \cite{NO89} to prove our result. 
Narendran-\'{O}'D\'{u}nlaing studied the decision problem on the finitely presented semigroup, and prove the following theorem.

\begin{theorem}[{\cite{NO89}*{Theorem 5.8}}]\label{thm:no}
    The question of whether a finitely presented commutative semigroup is cancellative is decidable. 
\end{theorem}

In \Cref{sec:commutative-monoid}, we construct a semigroup $M(S,\phi)$ associated to an interaction $(S,\phi)$. 
Let $S=\{1,\dots,n\}$. 
We define the semigroup $M(S,\phi)$ as 
\begin{equation*}
    M(S,\phi):=\langle a_1,\dots,a_n\mid a_ia_j=a_ka_l,\,\forall ((i,j),(k,l))\in\phi\rangle.
\end{equation*}
When the interaction is exchangeable, this semigroup $M(S,\phi)$ is commutative.  
We will reduce the question of whether the interaction $(S,\phi)$ is irreducibly quantified to the question of whether the semigroup $M(S,\phi)$ is cancellative and power-cancellative. 
More precisely, we will prove the following theorem. 
Combining this proposition with \Cref{thm:no}, we conclude our main result \Cref{thm:decidable-intro}. 

\begin{theorem}[{\Cref{thm:ccpcs-is-irred-q}}]\label{thm:ccpcs-is-irred-q-intor}
 A separable, exchangeable interaction $(S,\phi)$ is irreducibly quantified if and only if the commutative semigroup $M(S,\phi)$ is cancellative and power-cancellative. 
\end{theorem}

The construction of this paper is as follows. 
In \Cref{sec:interaction}, we recall the definition and some properties of interactions.
Next, in \Cref{sec:commutative-monoid}, we will introduce the commutative semigroup associated with interactions. 
We recall the basic properties of commutative semigroups. 
Finally, we prove our main result \Cref{thm:decidable-intro}.
At the end of this paper, we give a list of irreducibly quantified interactions for $S=\{0,1,2,3,4\}$. 

%
\section{Interaction}\label{sec:interaction}
%

In this section, we recall the definition of interactions. 

\begin{definition}\label{def:interaction}
Let $S$ be a finite set. 
\begin{enumerate}
\item We define a \textit{interaction} on $S$ to be a pair $(S,\phi)$ consisting of a subset $\phi\subset (S\times S)\times (S\times S)$ such that $(S\times S,\phi)$ form a symmetric digraph. 
\item We say that an interaction $(S,\phi)$ is \textit{exchangeable} if for any $(s,t)\in S\times S$, $(s,t)$ and $(t,s)$ are in the same connected component of $(S\times S,\phi)$. 
\end{enumerate}
\end{definition}

We first prove the decidability of exchangeability. 

\begin{proposition}\label{prop:exchangeable-is-decidable}
The question of whether an interaction $(S,\phi)$ is exchangeable is decidable. 
\end{proposition}

\begin{proof}
For any $(s,t)\in S\times S$, we construct a set $S_{s,t}\subset S\times S$ as follows.
\begin{description}
    \item[Step 1] Let $S_{s,t}^0=\{(s,t)\}$. 
    \item[Step 2] For any $i\in \mathbb{Z}$ and $(u,v)\in S_{s,t}^i$, let $T_{u,v}^i:=\{(u',v')\mid ((u,v),(u',v'))\in\phi\}$. We define $S_{s,t}^{i+1}=\bigcup_{(u,v)\in S_{s,t}^i}T_{s,t}^i$. 
\end{description}
Let $S_{s,t}:=\bigcup_{i=0}^\infty S_{s,t}^i$.
Since $S$ is finite, there is $N\in\mathbb{N}$ such that $S_{s,t}^N=S_{s,t}$. 
Then $S_{s,t}^N$ is a set of elements in $S\times S$ which are in the connected component containing $(s,t)$ in $(S\times S,\phi)$. 
Therefore we can check whether an element $(t,s)\in S\times S$ in the connected component containing $(s,t)$ in $(S\times S,\phi)$ by checking whether $(t,s)\in S_{s,t}^N$. 
Therefore, since $S\times S$ is finite, the question of whether an interaction $(S,\phi)$ is exchangeable is decidable.  
\end{proof}

In the study of the hydrodynamic limit, we are interested in the conserved quantity of an interaction. 
We fix a base point $*\in S$.
Then the conserved quantity is defined as follows. 

\begin{definition}\label{def:conserved-quantity}
\begin{enumerate}
    \item We define a \textit{conserved quantity} to be a function $\xi:S\rightarrow \mbb{R}$ such that for any $((s,t),(s',t'))\in \phi$, we have $\xi(s)+\xi(t)=\xi(s')+\xi(t')$. 
    We let $\Map^\phi(S,\mbb{R})$ denote the $\mbb{R}$-vector space of conserved quantities.
    \item We let $\Consv^\phi(S)$ denote the space of conserved quantities modulo the space of constant functions.
    By abuse of notation, for any $\xi\in \Consv^\phi(S)$ we also let $\xi$ denote its representative.
\end{enumerate}
\end{definition}

\begin{remark}\label{rem:basis}
    Since $\Map(S,\mbb{R})\cong \mbb{R}^S$ and $S$ is finite, $\mbb{R}$-vector subspace $\Map^\phi(S,\mbb{R})$ is finite dimensional. 
    Moreover, the basis of these vector spaces can be computed by solving a system of linear equations
    \begin{equation*}
        x_s+x_t=x_{s'}+x_{t'}\quad ((s,t),(s',t'))\in \phi.
    \end{equation*}
    In particular, we can obtain a basis of the spaces $\Map^\phi(S,\mbb{R})$ whose components are integers. 
    Moreover, if you fix a state $*\in S$, a basis of subspace $\Map^\phi(S,\mbb{R})$ consisting of functions $\xi$ such that $\xi(*)$ gives a basis of the $\mbb{R}$-vector space $\Consv^\phi(S)$.
\end{remark}

\begin{definition}\label{def:separable}
We say that an interaction $(S,\phi)$ is \textit{separable} 
if for any $s,t\in S$ if $\xi(s)=\xi(t)$ for any $\xi\in \Consv^\phi(S)$, we have $s=t$. 
\end{definition}

Since a constant function $f:S\rightarrow \mbb{R}$ satisfies $f(s)=f(t)$ for any $s,t\in S$, the notion of separability does not depend on the choice of a representative of $\xi\in \Consv^\phi(S)$.
By using following 

\begin{lemma}\label{lem:separable-decide}
    The question of whether an interaction $(S,\phi)$ is separable is decidable. 
\end{lemma}

\begin{proof}
    Let $\{\xi_i\}_{i\in I}$ be a basis of $\Consv^\phi(S)$ computed as in \Cref{rem:basis}. 
    Since the space of conserved quantities are $\mbb{R}$-vector space, the interaction $(S,\phi)$ is separable if and only if for any $s,t\in S$ such that $s\neq t$, there exists $i\in I$ such that $\xi_i(s)=\xi_i(t)$. 
    Therefore we can check whether an interaction $(S,\phi)$ is separable by comparing values of conserved quantities $\{\xi_i\}_{i\in I}$ for each distinct pair $(s,t)\in S^2$. 
    Since the number of distinct pair $(s,t)\in S^2$ is finite and $I$ is finite, the question of whether an interaction $(S,\phi)$ is separable is decidable. 
\end{proof}

Let $(X,E)$ be a finite connected symmetric digraph. 
We note that $E=\emptyset$ when $|X|=1$.
We say that $(X,E)$ is \textit{complete} if $(x,y)\in E$ for any $x,y\in X$. 
For any interaction $(S,\phi)$ and finite connected symmetric digraph $(X,E)$, we let $S^X:=\Map(X,S)$.
For any $\eta\in S^X$ and $x\in X$, we let $\eta_x$ denote the point $\eta(x)$.
We define the subset $\Phi_E\subset S^X\times S^X$ by 
\begin{equation*}
   \Phi_E=\{(\eta,\eta')\in S^X\times S^X\mid \exists (x,y)\in E,\,((\eta_x,\eta_y),(\eta'_x,\eta'_y))\in \phi,\,\forall z\neq x,y,\,\eta_z=\eta'_z\}.
\end{equation*}
Since $(S\times S,\phi)$ is symmetric graph, the digraph $(S^X,\Phi_E)$ is also symmetric. 
Let $\xi$ be a conserved quantity of $(S,\phi)$. 
We define a map $\xi_X:S^X\rightarrow \mathbb{R}$ by 
\begin{equation*}
\xi_X(\eta):=\sum_{x\in X}\xi(\eta_x)
\end{equation*}
for any $\eta\in S^X$.
We define the irreducibly quantified interaction as follows.

\begin{definition}\label{def:irreq}
We say that an interaction $(S,\phi)$ is \textit{irreducibly quantified} if for any finite connected symmetric digraph $(X,E)$ and for any $\eta,\eta'\in S^X$, if $\xi_X(\eta)=\xi_X(\eta')$ for any conserved quantity $\xi\in\Consv^\phi(S)$, then $\eta$ and $\eta'$ are in the same connected component of $(S^X,\Phi_E)$.
\end{definition}

We can easily observe that an irreducibly quantified interaction is always exchangeable and separable as follows.

\begin{lemma}\label{lem:irred-q-implies-exchangeable}
If an interaction $(S,\phi)$ is irreducibly quantified, then $(S,\phi)$ is exchangeable and separable.
\end{lemma}

\begin{proof}
    Let $(S,\phi)$ be an irreducibly quantified interaction.
    Let $(X,E)$ be a finite connected symmetric digraph such that $|X|=1$. 
    Then we have $S^X=S$ and $\Phi_E=\emptyset$. 
    Therefore, since $(S,\phi)$ is irreducibly quantified, for any $s_1,s_2\in S$ such that $\xi(s_1)=\xi(s_2)$ for any $\xi \in \Consv^\phi(S)$, we have $s_1=s_2$.
    This shows that the interaction $(S,\phi)$ is separable.

    We next consider a complete finite symmetric digraph $(X',E')$ such that $|X'|=2$. 
    Then we have $S^{X'}=S^2$ and $\Phi_E=\phi$. 
    Therefore, for any $s,t\in S$, since $\xi_{X'}(s,t)=\xi_{X'}(t,s)$ for any $\xi\in \Consv^\phi(S)$, two configurations $(s,t)$ and $(t,s)$ are in the same connected component of $(S\times S,\phi)$. 
    Hence the interaction $(S,\phi)$ is exchangeable.
\end{proof}

For any $\eta\in S^X$ and $x,y\in X$, we define $\eta^{x,y}\in S^X$ to be 
\begin{equation*}
\eta_z^{x,y}:=\left\{\begin{array}{cc}
\eta_y & z=x \\
\eta_x & z=y \\
\eta_z & \mathrm{otherwise}.
\end{array}\right.
\end{equation*}
Finally, we prove the following lemma which will be used in the proof of key lemma \Cref{lem:path-and-equation}.

\begin{lemma}\label{lem:app-exchangeable}
    Let $(S,\phi)$ be an exchangeable interaction, and let $(X,E)$ be a finite connected symmetric digraph. 
    For any $\eta\in S^X$ and $x,y\in X$, there is a path from $\eta$ to $\eta^{x,y}$ in $(S^X,\Phi_E)$. 
\end{lemma}

\begin{proof}
    Since $(X,E)$ is connected, there is a path 
    \begin{equation*}
        x=x_0\rightarrow\cdots\rightarrow x_k=y
    \end{equation*}
    from $x$ to $y$ in $(X,E)$. 
    For any $i\in\{ 0,\dots,k-1\}$, we inductively define $\eta^i$ by $\eta^0:=\eta$ and $\eta^{i+1}:=(\eta^i)^{x_i,x_{i+1}}$.
    We note that $\eta_y^{k-1}=\eta_x$.
    Since $(S,\phi)$ is exchangeable, for any $i\in\{0,\dots,k-1\}$, there is a path from $\eta^i$ to $\eta^{i+1}$ in $(S^X,\Phi_E)$. 
    For any $i\in \{k,\dots,2k-2\}$, we next define $\eta^i$ as $\eta^{i}=(\eta^{i-1})^{x_{2k-i-2},x_{2k-i-3}}$.
    Then we have $\eta^{2k-2}=\eta^{x,y}$. 
    Since $(S,\phi)$ is exchangeable, for any $i\in\{k,\dots,2k-2\}$, there is a path from $\eta^i$ to $\eta^{i+1}$ in $(S^X,\Phi_E)$. 
    Therefore there is a path from $\eta$ to $\eta^{x,y}$ in $(S^X,\Phi_E)$. 
\end{proof}

%
\section{Commutative semigroup and interaction}\label{sec:commutative-monoid}
%

In this section, we will construct a commutative semigroup associated with an interaction, and prove that the question of whether an interaction is irreducibly quantified is decidable. 
We first recall the definition of a semigroup. 
We define a \textit{semigroup} $M=(M,\cdot)$ to be a pair consisting of a set $M$ and a binary operation $\cdot:M\times M\rightarrow M$ which satisfies $(a\cdot b)\cdot c=a\cdot (b\cdot c)$ for any $a,b,c\in M$.
We often write $ab:=a\cdot b$ for any $a,b\in M$. 
A semigroup $M$ is \textit{commutative} if $ab=ba$ for any $a,b\in M$.

\begin{definition}
Let $M$ be a commutative semigroup. 
\begin{enumerate}
    \item We say that an element $c\in M$ is \textit{cancellative} in $M$ if $ca=cb$ implies $a=b$ for any $a,b\in M$. 
    \item We say that $M$ is \textit{cancellative} if every elements in $M$ are cancellative. 
    \item We say that $M$ is \textit{power-cancellative} if $a^n=b^n$ implies $a=b$ for any $a,b\in M$ and $n\in\mathbb{Z}_{>0}$. 
\end{enumerate}
\end{definition}

For any monoids $M,N$, we say that a map $f:M\rightarrow N$ is a \textit{homomorphism} if for any $a,b\in M$, we have $f(ab)=f(a)f(b)$.
Now, we recall the following result.

\begin{proposition}[{\cite{CS}*{II,Corollary 7.4}}]\label{prop:embedding}
Let $M$ be a finitely generated, cancellative, power-cancellative commutative semigroup which has no identity element.
The elements of $M$ are separated by finitely many homomorphisms of $M$ into $\mathbb{N}$. 
In other words, $M$ is isomorphic to a sub-semigroup of $\mathbb{N}^d$ for some $d\in\mathbb{Z}_{>0}$.
\end{proposition}

\begin{proof}
In \cite{CS}*{II,Corollary 7.4}, this proposition is proved when the group of units of $M$ is trivial.
Then a disjoint union $M'=M\cup \{e\}$, where $e$ is an identity element of $M'$, is a commutative monoid. 
Every element $a\in M$ is not a unit in the commutative monoid $M'$. 
Therefore the group of units of $M'$ is trivial.
By applying \cite{CS}*{II,Corollary 7.4} to $M'$, we have an inclusion $M'\hookrightarrow \mathbb{N}^d$. 
This induces an inclusion $M\hookrightarrow \mathbb{N}^d$. 
\end{proof}

Let $S=\{1,\dots, n\}$, and let $(S,\phi)$ be an interaction. 
We define a semigroup $M(S,\phi)$ by
\begin{equation*}
    M(S,\phi):=\langle a_1,\dots,a_n\mid a_ia_j=a_ka_l,\,((i,j),(k,l))\in\phi\rangle.
\end{equation*}
In other words, $M(S,\phi)$ is a semigroup constructed by dividing a free semigroup on $\{a_1,\dots,a_n\}$ by the congruence generated by relations $\{a_ia_j=a_ka_l\mid ((i,j),(k,l))\in\phi\}$.
For any $a\in M(S,\phi)$, we define the length $\mrm{len}(a)$ to be the smallest length of the expressions of $a\in M(S,\phi)$ as a product of $a_1,\dots,a_n$.
Since relations of the semigroup $M(S,\phi)$ do not change the length of words, 
all expressions of $a\in M(S,\phi)$ as a product of $a_1,\dots,a_n$ have the same length. 
Therefore, for any $a,b\in M(S,\phi)$, we have $\mrm{len}(ab)=\mrm{len}(a)+\mrm{len}(b)$.
In particular, we have a homomorphism $\mrm{len}:M(S,\phi)\rightarrow \mathbb{N}$. 
Moreover, the semigroup $M(S,\phi)$ has no identity element. 

\begin{lemma}\label{lem:commuatative-and-exchangeable}
Let $(S,\phi)$ be an interaction. 
A semigroup $M(S,\phi)$ is commutative if and only if $(S,\phi)$ is exchangeable.    
\end{lemma}

\begin{proof}
If $(S,\phi)$ is exchangeable, for any $i,j\in S$, $(i,j)$ and $(j,i)$ are in the same connected component of $(S\times S,\phi)$. 
Therefore there is a path 
\begin{equation*}
(i,j)=(i_0,j_0)\rightarrow(i_1,j_1)\rightarrow\cdots\rightarrow(i_m,j_m)=(j,i)
\end{equation*}
in $(S\times S,\phi)$.
This implies that
\begin{equation*}
a_ia_j=a_{i_0}a_{j_0}=a_{i_1}a_{j_1}=\cdots =a_{i_m}a_{j_m}=a_ja_i
\end{equation*}
in $M(S,\phi)$. 
Therefore $M(S,\phi)$ is commutative.

Next, we assume that $M(S,\phi)$ is commutative.
For any $i',j'\in S$, we have $a_{i'}a_{j'}=a_{j'}a_{i'}$. 
Then we can get $a_{j'}a_{i'}$ by applying relations of $M(S,\phi)$ to $a_{i'}a_{j'}$ finitely many times. 
Since all relations of $M(S,\phi)$ does not change length of words, we have 
\begin{equation*}
    a_{i'}a_{j'}=a_{i'_0}a_{j'_0}=a_{i'_1}a_{j'_1}=\cdots =a_{i'_m}a_{j'_m}=a_{j'}a_{i'}
\end{equation*}
such that for any $k\in\{1,\dots,m\}$, an equality $a_{i'_k}a_{j'_k}=a_{i'_{k+1}}a_{j'_{k+1}}$ is a relation of $M(S,\phi)$. 
Therefore there is a path 
\begin{equation*}
(i',j')=(i'_0,j'_0)\rightarrow(i'_1,j'_1)\rightarrow\cdots\rightarrow(i'_m,j'_m)=(j',i')
\end{equation*}
in $(S\times S,\phi)$. 
\end{proof}

Next, we prove that $\Map^\phi(S,\mbb{R})$ and $\mathrm{Hom}(M(S,\phi),\mathbb{R}_+)$, where $\mathbb{R}_+$ is a monoid on $\mathbb{R}$ with its summation, are isomorphic $\mathbb{R}$-vector space. 

\begin{lemma}\label{lem:isom-of-vector-space}
Let $\xi\in \Consv^\phi(S)$. 
We define a map $\xi_M:M(S,\phi)\rightarrow \mathbb{R}_+$ by $\xi_M(a_i)=\xi(i)$ for any $i\in S$. 
This induces an $\mathbb{R}$-linear map 
\begin{equation*}
u:\Map^\phi(S,\mbb{R})\rightarrow \mathrm{Hom}(M(S,\phi),\mathbb{R}_+);\quad \xi\mapsto \xi_M.
\end{equation*}
Moreover, this map is an isomorphism.
\end{lemma}

\begin{proof}
Let $\xi\in \Map^\phi(S,\mbb{R})$. 
We first prove that $\xi_M$ is a well-defined homomorphism. 
By definition of $\xi_M$, for any relation $a_ia_j=a_ka_l$ of $M(S,\phi)$, 
we have 
\begin{equation*}
    \xi_M(a_ia_j)=\xi(i)+\xi(j)=\xi(k)+\xi(l)=\xi_M(a_ka_l).
\end{equation*}
Therefore the map $\xi_M$ is a well-defined homomorphism. 
Hence the assignment $\xi\mapsto \xi_M$ induces an $\mbb{R}$-linear map $\Map^\phi(S,\mbb{R})\rightarrow \mathrm{Hom}(M(S,\phi),\mathbb{R}_+)$.

We will prove that this map is an isomorphism. 
For any $\xi,\xi'\in \Map^\phi(S,\mbb{R})$, if for any $i\in S$, we have $\xi_M(a_i)=\xi'_M(a_i)$, then we have $\xi(a_i)=\xi'(a_i)$.
Therefore we have $\xi=\xi'$. 
Hence the map $u$ is injective.  
Next, we prove the surjectivity. 
For any $\psi\in \mathrm{Hom}(M(S,\phi),\mathbb{R})$, we define $\xi^\psi:S\rightarrow \mathbb{R}$ by $\xi^\psi(i)=\psi(a_i)$. 
For any $((i,j),(i',j'))\in \phi$, we have 
\begin{equation*}
\xi^\psi(i)+\xi^\psi(j)=\psi(a_ia_j)=\psi(a_{i'}a_{j'})=\xi^\psi(i')+\xi^\psi(j').
\end{equation*}
Therefore $\xi^\psi$ is a conserved quantity and $\xi_M^\psi=\psi$. 
Hence the map $u$ is surjective.
\end{proof}

Let $(S,\phi)$ be an exchangeable interaction, and let $(X,E)$ be a finite connected symmetric digraph. 
For any $\eta\in S^X$, we define an element $\eta_M$ of $M$ by 
\begin{equation*}
    \eta_M:=a_1^{\sharp \{x\in X\mid \eta_x=1\}}\cdots a_n^{\sharp \{x\in X\mid \eta_x=n\}}.
\end{equation*}
We prove that this assignment induces a surjection $S^X\mapsto \{a\in M(S,\phi)\mid l(a)=|X|\}$.

\begin{lemma}\label{lem:inter-to-semig-surj}
    The map $f:S\rightarrow M(S,\phi)$ induced by the assignment $\eta\mapsto \eta_M$ is surjective.
\end{lemma}

\begin{proof}
    Let $a_i^{p_i}\cdots a_n^{p_n}\in M(S,\phi)$. 
    We fix a bijection $\iota:X\rightarrow\{1,\dots,|X|\}$.
    We define $\eta\in S^X$ by $\eta_x=i$ for any $x\in X$ such that $\sum_{j=1}^{i-1}p_j< \iota(i)\leq \sum_{j=1}^{i}p_j$.
    Then we have $\eta_M=a_i^{p_i}\cdots a_n^{p_n}$. 
    Therefore the map $f$ is surjective. 
\end{proof}

The following lemma is key to reducing the problem for interactions to the problem of semigroups.

\begin{lemma}\label{lem:path-and-equation}
Let $(S,\phi)$ be an exchangeable interaction, and let $(X,E)$ be a finite connected symmetric digraph. 
For any $\eta,\eta'\in S^X$, $\eta$ and $\eta'$ are in the same connected component of $(S^X,\Phi_E)$ if and only if $\eta_M=\eta'_M$ in $M(S,\phi)$.
\end{lemma}

\begin{proof}
Let $\eta,\eta'\in S^X$. 
We assume that $\eta$ and $\eta'$ are in the same connected component of $(S^X,\Phi_E)$. 
Then there is a path 
\begin{equation*}
\eta=\eta^0\rightarrow\cdots\rightarrow\eta^m=\eta'
\end{equation*}
in $(S^X,\Phi_E)$. 
For any $i\in\{1,\dots,m\}$, there exist $x_i, y_i\in X$ such that $((\eta^i_{x_i},\eta^i_{y_i}),(\eta^{i+1}_{x_i},\eta^{i+1}_{y_i}))\in \phi$ and $\eta_z^i=\eta_z^{i+1}$ for any $z\neq x_i,y_i$. 
Since there is a relation $a_{\eta^i_{x_i}}a_{\eta^i_{y_i}}=a_{\eta^{i+1}_{x_i}}a_{\eta^{i+1}_{y_i}}$ of $M(S,\phi)$, we have
\begin{equation*}
\eta_M=\eta_M^0=\dots=\eta_M^m=\eta'_M 
\end{equation*}
in $M(S,\phi)$. 

On the other hand, we assume that $\eta_M=\eta'_M$.
Then we have equalities 
\begin{equation*}
\eta_M=b_0=\cdots=b_{m'}=\eta'_M
\end{equation*}
such that, for any $i\in \{0,\dots, m'-1\}$, if we write $b_i=a_1^{p_{i,1}}\cdots a_n^{p_{i,n}}$, there exist pairs $(j_i,k_i),(j'_i,k'_i)\in \{1,\dots,n\}^2$ such that 
\begin{equation*}
p_{i+1,j}=p_{i,j}-\delta_{j_i}-\delta_{k_i}+\delta_{j'_i}+\delta_{k'_i}, \quad \delta_k(l):=\begin{cases}
    1 & k=l,\\
    0 & k\neq l,
\end{cases}
\end{equation*}
and $a_{j_i}a_{k_i}=a_{j'_i}a_{k'_i}$ is a relation of $M(S,\phi)$. 
Let $\eta^0:=\eta$. 
For any $i\in\{1,\dots,m'\}$, by applying following construction ($\heartsuit$), we define $\eta^i$ inductively:
\begin{description}
    \item[($\heartsuit$)] If a given $\eta^i$ satisfies $\eta_M^i=b_i$, we can take $x_i,y_i\in X$ such that 
    \begin{equation*}
        \eta_{x_i}^i=j_i,\quad \eta_{y_i}^i=k_i.
    \end{equation*}
    We define $\eta^{i+1}$ by 
    \begin{equation*}
        \eta_z^{i+1}=\left\{
        \begin{array}{cc}
        j'_i&z=x_i\\
        k'_i&z=y_i\\
        \eta_z^i&z\neq x_i,y_i.
        \end{array}
        \right.
    \end{equation*}
    Since $\eta^{i+1}$ satisfies $\eta_M^{i+1}=b_{i+1}$, if $i+1\in \{1,\dots,m'-1\}$, then we can apply ($\heartsuit$) to $\eta^{i+1}$.  
\end{description}
We note that $\eta_M^{m'}=\eta'_M$.
For any $i\in \{1,\dots,m'\}$, take $z_i\in X$ such that $(z_i,y_i)\in E$.
Since $(S,\phi)$ is exchangeable, by \Cref{lem:app-exchangeable}, there is a path from $\eta^i$ to $(\eta^i)^{x_i,z_i}$ in $(S^X,\Phi_E)$. 
Then, by the construction, we have $((\eta^i)^{x_i,z_i},(\eta^{i+1})^{x_i,z_i})\in \Phi_E$. 
By applying \Cref{lem:app-exchangeable} again, there is a path from $\eta^i$ to $\eta^{i+1}$ in $(S^X,\Phi_E)$.
Since $\eta_M^{m'}=\eta'_M$, by applying \Cref{lem:app-exchangeable} finitely many times, we obtain a path from $\eta^{m'}$ to $\eta'$ in $(S^X,\Phi_E)$. 
Therefore we conclude that $\eta$ and $\eta'$ are in the same connected component of $(S^X,\Phi_E)$. 
\end{proof}

Now, we reduce the irreducible quantifiability of interactions to the cancellativity and power-cancellativity of the semigroups induced by them, as follows.

\begin{theorem}\label{thm:ccpcs-is-irred-q}
 A separable, exchangeable interaction $(S,\phi)$ is irreducibly quantified if and only if the commutative semigroup $M(S,\phi)$ is cancellative and power-cancellative. 
\end{theorem}

\begin{proof}
We assume that $(S,\phi)$ is irreducibly quantified. 
We first prove that $M(S,\phi)$ is cancellative. 
Let $a,b,c\in M(S,\phi)$ such that $ac=bc$.
By \Cref{lem:isom-of-vector-space}, for any $\xi\in\mathrm{Map^\phi(S,\mbb{R})}$, since $\xi_M(ac)=\xi_M(bc)$, we have $\xi_M(a)=\xi_M(b)$. 
In particular, if we let $\xi$ be a constant function, it induces $\mrm{len}(a)=\mrm{len}(b)$.
If $\mrm{len}(a)=1$, since $(S,\phi)$ is separable, we have $a=b$. 
Therefore we may assume that $\mrm{len}(a)\geq 2$. 
Let $(X,E)$ be a complete finite symmetric digraph of size $\mrm{len}(a)$.
By \Cref{lem:inter-to-semig-surj}, there exist $\eta^a$ and $\eta^b\in S^X$ such that $\eta_M^a=a$ and $\eta_M^b=b$, respectively.  
For any $\xi\in \Map^\phi(S,\mbb{R})$, since $\xi_M(a)=\xi_M(b)$, we have $\xi_X(\eta^a)=\xi_X(\eta^b)$. 
Therefore, since $(S,\phi)$ is irreducibly quantified, $\eta^a$ and $\eta^b$ are in the same connected component of $(S^X,\Phi_E)$. 
By \Cref{lem:path-and-equation}, we have $a=b$. 
Therefore $M(S,\phi)$ is cancellative. 

We next prove that $M(S,\phi)$ is power-cancellative. 
Let $a,b\in M(S,\phi)$. 
We assume that there is $N\in\mathbb{N}$ such that $a^N=b^N$. 
By \Cref{lem:isom-of-vector-space}, we have $\xi_M(a^N)=\xi_M(b^N)$ for any $\xi\in \Map^\phi(S,\mbb{R})$. 
Therefore we have $\xi_M(a)=\xi_M(b)$ for any $\xi\in \Map^\phi(S,\mbb{R})$.
By a similar argument as above, we have $a=b$. 
Therefore $M(S,\phi)$ is power-cancellative. 

On the other hand, we next assume that $M(S,\phi)$ is cancellative and power-cancellative. 
Let $(X,E)$ be a finite connected symmetric digraph, and let $\eta,\eta'\in S^X$ such that $\xi_x(\eta)=\xi_X(\eta')$ for any $\xi\in \Map^\phi(S,\mbb{R})$. 
By \Cref{prop:embedding}, we have an inclusion $\pi:M(S,\phi)\hookrightarrow \mathbb{N}^d$.
For each $i=1,\dots,d$, we let $\pi_i$ denote the projection $\mathbb{N}^d\rightarrow \mathbb{N}$ on the $i$-th component. 
Then the composition $\pi_i\circ \pi$ is a homomorphism $M(S,\phi)\rightarrow \mathbb{N}\subset \mbb{R}_+$. 
By \cref{lem:isom-of-vector-space}, for any $i\in\{1,\dots,d\}$, there is $\xi^i$ such that $\xi_M^i=\pi_i\circ \pi$.
Therefore we have 
\begin{equation*}
\pi_i\circ \pi(\eta_M)=\xi_X^i(\eta)=\xi_X^i(\eta')=\pi_i\circ \pi(\eta'_M)    
\end{equation*}
for any $i\in\{1,\dots,d\}$. 
Since $\pi$ is an inclusion, we have $\eta_M=\eta'_M$. 
Therefore, by \Cref{lem:path-and-equation}, $\eta$ and $\eta'$ are in the same connected component in $(S^X,\Phi_E)$, and hence we conclude that $(S,\phi)$ is irreducibly quantified. 
\end{proof}

The main theorem of this paper is as follows. 

\begin{theorem}\label{thm:decidable}
 The question of whether an interaction $(S,\phi)$ is irreducibly quantified is \textbf{decidable}. 
\end{theorem}

\begin{proof}
Let $(S,\phi)$ be an interaction. 
By \Cref{prop:exchangeable-is-decidable} and \Cref{lem:separable-decide}, we can decide whether $(S,\phi)$ is exchangeable and separable. 
By \Cref{lem:irred-q-implies-exchangeable}, if the interaction $(S,\phi)$ is not exchangeable or separable, it cannot be irreducibly quantified. 
Therefore we may assume that the interaction $(S,\phi)$ is exchangeable and separable in the following. 

By \Cref{thm:ccpcs-is-irred-q}, it suffices to show that the question of whether, for an exchangeable, separable interaction $(S,\phi)$, the induced semigroup $M(S,\phi)$ is cancellative and power-cancellative.  
By \cite{NO89}*{Theorem 5.8}, in general, the question of whether a finitely presented commutative semigroup is cancellative is decidable. 
Therefore the question of whether a commutative semigroup $M(S,\phi)$ is cancellative is decidable. 

A cancellative commutative semigroup $M$ can be embedded into the Grothendieck group 
\begin{equation*}
G(M)=\mathbb{Z}^n/\{\mathbbm{e}_i+\mathbbm{e}_j-\mathbbm{e}_{i'}-\mathbbm{e}_{j'}=0\mid ((i,j),(i',j'))\in \phi\},
\end{equation*}
where $\mathbbm{e}_i$ for $i\in \{1,\dots,n\}$ are standard basis of $\mathbb{Z}^n$.  
By \cite{CS}*{II, Proposition 5.4}, a commutative semigroup $M$ is power-cancellative if and only if $G(M)$ is torsion-free. 
The question of whether the abelian group $G(M)$ is torsion-free is decidable by computing the Smith normal of $n\times |\phi|$-matrix
\begin{equation*}
A=(\mathbbm{e}_i+\mathbbm{e}_j-\mathbbm{e}_{i'}-\mathbbm{e}_{j'})_{((i,j),(i',j'))\in \phi}.
\end{equation*}
Therefore we conclude that the question of whether an exchangeable interaction $(S,\phi)$ is irreducibly quantified is decidable. 
\end{proof}

%
\section{Appendix: List of irreducibly quantified interactions}\label{sec:appendix}
%

In this section, we give a list of irreducibly quantified interactions for $S = \{0,1,2,3,4\}$. 
Source codes for this computation can be found in \cite{Irrq} and \cite{List}.
We introduce the concept of equivalence of interactions. 

\begin{definition}
    Let $(S,\phi)$ and $(S',\phi')$ are interactions.
    We say that $(S,\phi)$ and $(S',\phi')$ are \textit{equivalent} if there exists a bijection $S\cong S'$ such that the induced map $\Map(S',\mbb{R})\cong \Map(S,\mbb{R})$ induces an $\mbb{R}$-linear isomorphism $\Consv^{\phi'}(S')\cong \Consv^\phi(S)$.
\end{definition}

In this list of \Cref{fig:irreq} on \Cpageref{fig:irreq}, we identify interactions which are equivalent to each other. 
In other words, each irreducibly quantified interaction for $S = \{0,1,2,3,4\}$ is equivalent to one of the interactions in \Cref{fig:irreq}. 
Each caption represents a basis of conserved quantities of the interaction. 
For any $i=1,\dots,5$, $\xi_i:S\rightarrow \mbb{R}$ is the function such that $\xi_i(j)=\delta_{ij}$.

\begin{remark}
    In general, the equivalence of interaction does not reflect the irreducibly quantifiedness.
    For example, for the interaction \Cref{fig:irreq} (D), if we remove edges $(2,2)\leftrightarrow (0,4), (4,0)$ and add an edge $(0,4)\leftrightarrow (4,0)$, we obtain the new exchangeable, separable interaction $(S,\phi')$ which is equivalent to the interaction \Cref{fig:irreq} (D) by the identity map.
    However, this interaction $(S,\phi')$ is not irreducibly quantified. 
    In fact, if $(X,E)$ is a complete symmetric digraph on $2$ vertices and $\eta=(2,2), \eta'=(0,4)\in S^X$, then we have $\xi_X(\eta)=\xi(\eta')$ for any $\eta\in \Consv^{\phi'}(S)$, however $\eta$ and $\eta'$ are in different connected components of $(S^X,\Phi'_E)$ which is a graph on configuration space associated to $(S,\phi')$.
\end{remark}

We also give a list of interactions for $S=\{0,1,2,3,4\}$ which are exchangeable and separable but not irreducibly quantified.
See \Cref{fig:not-irreq} on \Cpageref{fig:not-irreq}. 
Each of them has a counterexample on the complete symmetric digraph $(X,E)$ on $3$ vertices. 
In fact, the counterexamples are 
\begin{description}
    \item[\Cref{fig:not-irreq} (A)] $\eta=(0,4,4)$ and $\eta'=(2,2,2)$, 
    \item[\Cref{fig:not-irreq} (B)] $\eta=(0,3,4)$ and $\eta'=(1,1,1)$, 
    \item[\Cref{fig:not-irreq} (C)] $\eta=(0,2,3)$ and $\eta'=(1,1,4)$.
\end{description}
By easy computation, we can check that for each pair $(\eta,\eta')$, we have $\xi_X(\eta)=\xi_X(\eta')$ for any $\xi\in \Consv^\phi(S)$, but they are in different connected components of $(S^X,\Phi_E)$.

\begin{remark}
    For the case $S=\{0,1,2\}$ and $S=\{0,1,2,3\}$, any interaction which is exchangeable and separable is irreducibly quantified. 
    The proof and more detailed classifications for these cases are provided in the forthcoming paper \cite{BKSWY}.
\end{remark}

%
\section*{Acknowledgement}
%
The author is grateful to my advisor Kenichi Bannai for his support throughout the project.
We would like to thank members of the Hydrodynamic Limit Seminar at Keio/RIKEN. 
In particular, I would like to thank Jun Koriki for the advice he provided while writing the program to list up the irreducibly quantified interactions.
This work was supported by RIKEN Junior Research Associate Program.

\begin{figure}[b]
\centering
\input{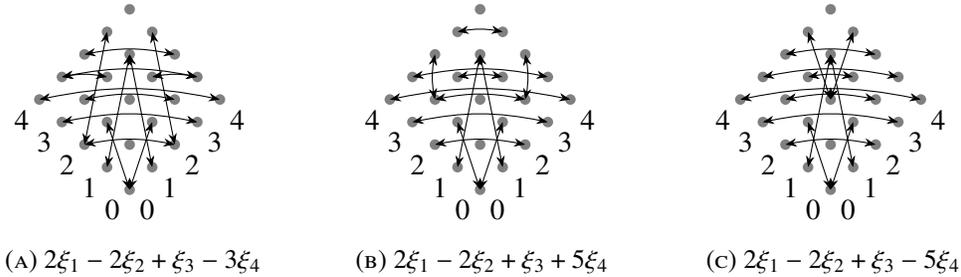}
\caption{Interactions for $S = \{0,1,2,3,4\}$ which is exchangeable and separable but not irreducibly quantified.}
\label{fig:not-irreq}
\end{figure}

\begin{figure}[p]
\centering
\input{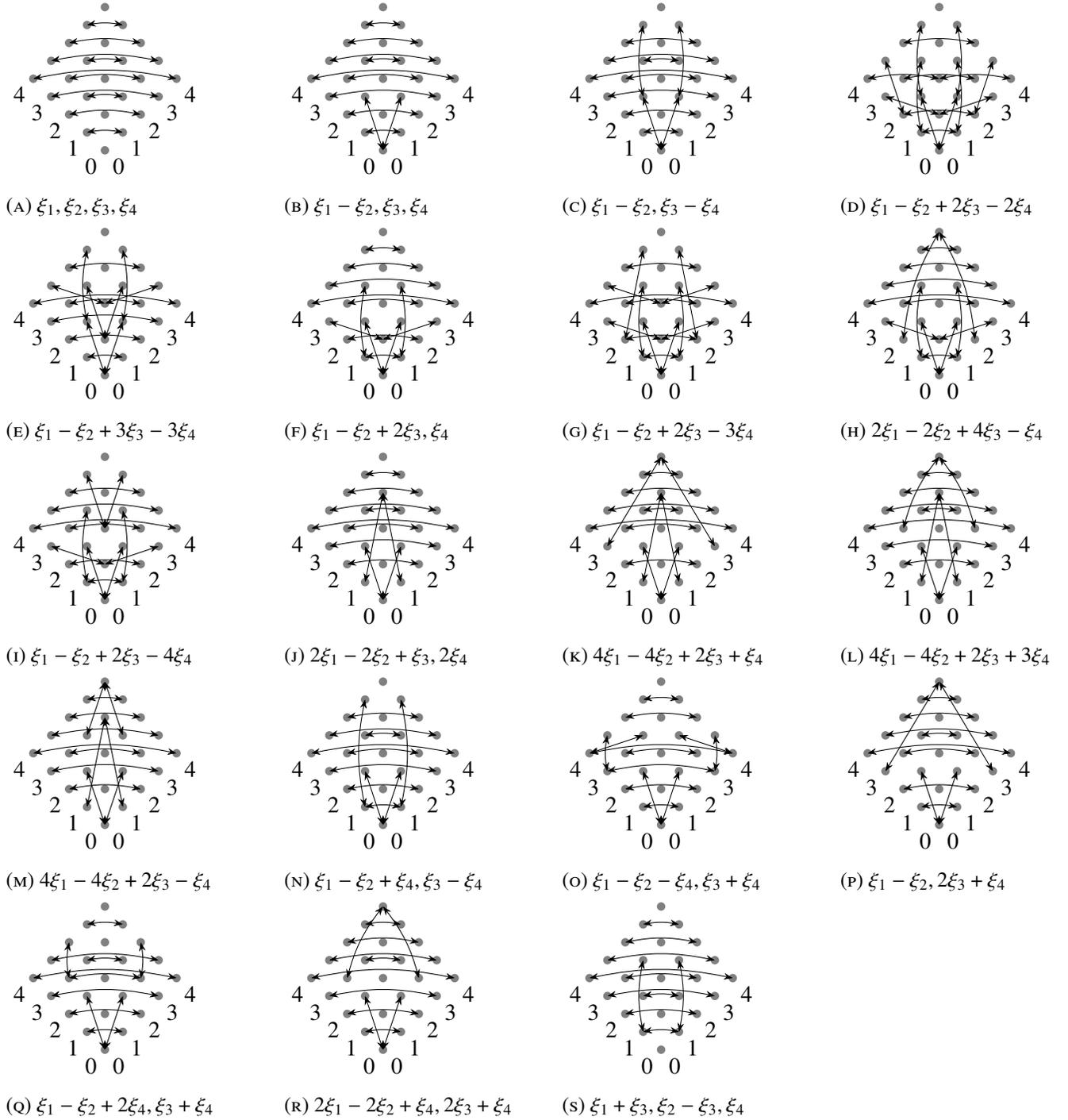}
\caption{Irreducibly quantified interactions for $S = \{0,1,2,3,4\}$.}
\label{fig:irreq}
\end{figure}

\end{document}